\documentclass[a4paper,11pt]{article}

\usepackage{float}
\usepackage{amsmath}
\usepackage{amsthm}
\usepackage{amsfonts}
\usepackage{amssymb}
\usepackage[T1]{fontenc}
\usepackage{geometry}
\usepackage{graphicx}
\usepackage{hyperref}
\usepackage{icomma}
\usepackage[latin1]{inputenc}
\usepackage{latexsym}
\usepackage[numbers]{natbib}
\usepackage{textcomp}
\usepackage[all]{xy}
\usepackage{cancel}
\usepackage{dsfont}

\title{Exponential mixing properties for time inhomogeneous diffusion processes with killing}
\author{Pierre Del Moral\thanks{School of Maths and Stats, UNSW Sydney, Australia},
\and
Denis Villemonais
\thanks{TOSCA project-team, INRIA Nancy -- Grand Est, IECL -- UMR 7502, Universit\'e de Lorraine
B.P. 70239, 54506 Vandoeuvre-l\`es-Nancy Cedex, France }
}

\DeclareMathSymbol{\minus}{\mathord}{operators}{"2D}

\newtheorem{theorem}{Theorem}
\newtheorem{lemme}[theorem]{Lemma}
\newtheorem{proposition}[theorem]{Proposition}

\theoremstyle{remark}

\newtheorem{remark}{Remark}

\def\be{\begin{eqnarray}}
\def\ee{\end{eqnarray}}
\def\ben{\begin{eqnarray*}}
\def\een{\end{eqnarray*}}
\def\bei{\begin{itemize}}
\def\eei{\end{itemize}}
\def\me{\medskip \noindent}
\def\bi{\bigskip \noindent}
\def\E{\mathbb{E}}
\def\P{\mathbb{P}}
\def\R{\mathbb{R}}
\def\1{\mathbf{1}}

\def\d{\partial}

\begin{document}
\maketitle

\begin{abstract}
We consider an elliptic and time-inhomogeneous
diffusion process with time-periodic coefficients evolving in a bounded domain of $\mathbb{R}^d$ with a smooth boundary. 
The process is killed when it hits the boundary of the domain (hard killing) or after an exponential time (soft killing) associated with some bounded rate function. 
The branching particle interpretation of the non absorbed diffusion again behaves as a set of interacting particles evolving in an absorbing medium. Between absorption times, the particles
evolve independently one from each other according to the diffusion evolution operator; when a particle is absorbed, another selected particle splits into two offsprings.
This article is concerned with the stability properties of these non absorbed processes. Under some classical ellipticity properties on the diffusion process and 
some mild regularity properties of the hard obstacle boundaries, we prove an uniform exponential 
strong mixing property of 
  the process conditioned  to not be killed. We also provide uniform estimates w.r.t. the time horizon for the interacting particle interpretation of these non-absorbed processes, yielding what seems to be the first result of this type for this class of diffusion processes evolving in soft and hard obstacles, both in homogeneous and non-homogeneous time settings.
\end{abstract}



\section{Introduction}
Let $D$ be a bounded open subset of $\mathds{R}^{d}$ ($d\geq 1$) whose boundary $\partial D$ is of class $C^2$ and consider the stochastic differential equation
\begin{equation}
  \label{eq:the-eds}
  dZ_t=\sigma(t,Z_t) dB_t+ b(t,Z_t) dt,\ Z_0\in D,
\end{equation}
where $B$ is a standard $d$ dimensional Brownian motion. We assume that the functions
\begin{equation*}
  \sigma:
  \begin{array}{l}
    [0,+\infty[\times\mathds{R}^d\rightarrow \mathds{R}^d\times\mathds{R}^d\\
    (t,x)\mapsto \sigma(t,x)
  \end{array}
  \text{ and }
  b:
  \begin{array}{l}
    [0,+\infty[\times\mathds{R}^d\rightarrow \mathds{R}^d\\
    (t,x)\mapsto b(t,x)
  \end{array}
\end{equation*}
 are continuous on $[0,+\infty[\times\R^d$. Moreover, we assume that they are time-periodic and Lipschitz in $x\in D$ uniformly in $t\in[0,+\infty[$. This means that there exist two constants  $\Pi> 0$ and $C_0>0$ such that, for all $x,y\in D$ and $t\geq 0$,
   \begin{align}
   \label{eq:assumption1}
     &\sigma(t+\Pi,x)=\sigma(t,x)\text{ and }b(t+\Pi,x)=b(t,x),\nonumber\\
     &\|\sigma(t,x)-\sigma(t,y)\|+|b(t,x)-b(t,y)|\leq k_0 |x-y|.
   \end{align}
In particular, there exists a solution to the stochastic differential equation~\eqref{eq:the-eds} (see \cite[Theorem 3.10, Chapter 5]{Ethier1986}). Moreover,
the solution is path-wise unique up to time $\tau_D=\inf\{t\geq
0,\ Z_t\notin D\}$ (see \cite[Theorem 3.7, Chapter 5]{Ethier1986}). Note that, since $\sigma,b$ are continuous in the compact set $[0,\Pi]\times \overline{D}$ and since they are time-periodic, both functions $\sigma$ and $b$ are uniformly bounded.

\me For all $s>0$ and any probability
distribution $\mu$ on $D$, we denote by $({\cal Z}^{\mu}_{s,t})_{t\geq
  s}$ the unique solution to this stochastic differential equation
starting at time $s>0$ with distribution $\mu$, killed when it hits the boundary and killed with a rate
$\kappa(t,{\cal Z}^{x}_{s,t})\geq 0$, where 
\begin{equation*}
\kappa: [0,+\infty[\times D \to \R_+
\end{equation*}
is a uniformly bounded non-negative measurable function which is also time-periodic  (with period $\Pi$). By \textit{``the process is
  killed''}, we mean that the process is sent to a cemetery point
$\partial \notin D$, so that the killed process is c\`adl\`ag almost
surely. If there exists $x\in D$ such that
$\mu=\delta_x$, we set $({\cal Z}^{x}_{s,t})_{t\geq s}=({\cal
  Z}^{\delta_x}_{s,t})_{t\geq s}$.
 When the process is killed by hitting the boundary, we say that
it undergoes a \textit{hard} killing; when the process is killed strictly
before reaching the boundary (because of the killing rate
$\kappa$), we say that it undergoes a \textit{soft} killing. We
denote the killing time by $\tau_{\partial}=\inf\{t\geq s,\,{\cal
  Z}^{x}_{s,t}=\partial\}$.

\medskip \noindent 
Let $(Q_{s,t})_{0\leq s \leq t}$ be the evolution operator of the diffusion process with killing. It is defined, for all $0\leq
s\leq t$, $x\in D$ and for any bounded measurable function
$f:\mathbb{R}^d\cup \{\partial\}\mapsto \mathbb{R}$ which vanishes
outside $D$, by
 \begin{eqnarray*}
   Q_{s,t} f(x)=\E\left(f({\cal Z}^x_{s,t})\1_{t<\tau_{\partial}}\right)=\E\left(f({\cal Z}^x_{s,t})\right).
 \end{eqnarray*}
We emphasize that, for any probability measure $\mu$ on $D$, the law
of ${\cal Z}^{\mu}_{s,t}$ is given by the probability measure $\mu
Q_{s,t}$, defined by
\begin{equation*}
  \mu Q_{s,t}(f) = \int_{D} Q_{s,t} f(x) d\mu(x).
\end{equation*}

\bigskip \noindent In this paper, we focus on the long time
behaviour of the distribution of ${\cal Z}^{\mu}_{s,t}$ conditioned to not be killed when it is observed, that is to the event 
$\{t<\tau_{\partial}\}$. This distribution is given by
\begin{equation}
  \label{eq:the-conditioned-distribution}
  \P({\cal Z}^{\mu}_{s,t}\in\cdot\,|\,t<\tau_{\partial})=\frac{\mu Q_{s,t}(\cdot)}{\mu Q_{s,t} (\mathbf{1})}.
\end{equation}

The first aim of this paper is to provide a sufficient criterion for the following exponential
    mixing property : 
there
 exist two constants $C>0$ and $\gamma>0$ such that, for any
 probability measure $\mu_1$ and $\mu_2$ on $D$,
 \begin{equation}
   \label{eq:the-mixing-property}
   \left\|\frac{\mu_1 Q_{s,t}}{\mu_1 Q_{s,t}
     (\1_D)}-\frac{\mu_2 Q_{s,t}}{\mu_2 Q_{s,t}
     (\1_D)}\right\|_{TV}\leq C e^{-\gamma (t-s)},
 \end{equation}
 where $\|\cdot\|_{TV}$ denotes the total variation norm between measures on $\R^d$.

The usual tools to prove convergence as in~\eqref{eq:the-mixing-property} for non-conditioned evolution operator involve coupling arguments: for example, contraction in total
variation norm  can be obtained using mirror and parallel coupling see~\cite{LindvallRogers1986,Wang2004,Priola2006}. However, the exponential contraction obtained for the non-conditioned evolution operator is not maintained after conditioning, because the denominator $\mu Q_{s,t} (\mathbf{1}_D)$ goes to $0$ when $t\rightarrow\infty$, usually as a stronger pace than the rate of contraction obtained by coupling methods for the unconditioned evolution operator.

In order to overcome this difficulty, our main tool will be to consider, for any time horizon $T>0$, the process conditioned not to be killed before time $T$ and to prove uniform rate of contraction in total variation norm for this collection of time-inhomogeneous processes. However, this can not be obtained directly from existing coupling methods, since the process conditioned not to be killed up to a
given time $t>0$ is a time-inhomogeneous diffusion process with a singular drift for which the traditional coupling methods mentioned above fail. For instance, a
standard $d$-dimensional Brownian motion $(B_t)_{t\geq 0}$ conditioned not to exit a smooth domain $D\subset\R^d$ up to a time $T>0$
has the law of the solution $(X^{(T)}_t)_{t\in[0,T]}$ to the stochastic differential equation
\begin{align*}
dX^{(T)}_t=dB_t+\left[\nabla \ln Q_{t,T}\1_D((X^{(T)}_t))\right]dt.
\end{align*}
Since the probability of not being killed $Q_{t,T}\1_D(x)$ vanishes when $x$ converges to the boundary $\partial D$ of $D$, the drift term in the above SDE is
singular and hence existing coupling methods do not apply directly. Moreover, an additional difficulty comes from the fact that the coupling coefficients must be controlled uniformly in the time horizon $T>0$, while the drift term in the above equation changes with $T$.

While the idea to consider the process conditioned to not be killed before a varying time horizon is not new when processes are only subject to a uniformly bounded soft killing rate (see~\cite{dp-2013}), it has never been used when a hard killing rate is involved because of the singularities mentioned above. Instead, convergence of conditioned diffusion processes with hard killing have been
obtained up to now using (sometimes involved) spectral theoretic arguments (see for instance, \cite{Cattiaux2009,Kolb2011,Littin2012,Miura2014} for
one-dimensional diffusion processes and \cite{Cattiaux2009,KnoblochPartzsch2010} for multi-dimensional diffusion processes) which all require at least $C^1$ diffusion coefficients and become impractical in the time-homogeneous setting.

In this paper, we use a new approach to obtain uniform coupling for the process conditioned to not be killed before a varying time horizon $T$ using existing coupling results for non-conditioned processes and fine control over the probability of killing before time $T$. Doing so, our objective is twofold : first we aim at relaxing the regularity assumptions required by the classical spectral methods, second we provide tools that apply to time-inhomogeneous processes with hard killing, which were not the case of previous approaches.

Note that our approach is entirely based on probabilistic tools which leads us to original results in both time-homogeneous and time-inhomogeneous cases. Let us first discuss the novelty of our results in the time-homogeneous situation. The most advanced results on the matter are given by Knobloch and Partzsch in~\cite{KnoblochPartzsch2010} and are obtained by spectral theory tools, whose main drawback is the essential requirement of regularity for the infinitesimal generator. More precisely, they require $\sigma$ to be of class $C^1$ on the whole domain $D$ in order to conclude. On the contrary, we only assume differentiability of $\sigma$ in an arbitrary neighbourhood of the boundary $\d D$, while the coefficients are only required to be Lipschitz in the rest of the domain. In fact, our criterion also applies to situations where the coefficient $\sigma$ is only Lipschitz in the whole domain $D$. We give several examples in the next section, where our criterion applies while existing results from~\cite{CattiauxMeleard2010}, \cite{KnoblochPartzsch2010}, \cite{Pinsky1985}, \cite{Gong1988} fail because of their uniform regularity requirements. 

Let us now comment the time-inhomogeneous setting. In this case, the classical spectral theory approach cannot be used since the infinitesimal generator  depends on time. Also, since the coefficients are assumed to be time-periodic, it is natural to ask whether discrete time results could be applied in this situation. However, no existing tools can handle a situation where the killing probability between two successive time is not uniformly bounded away from $0$, as it is in our case of a diffusion process with a hard boundary. On the contrary, our approach, by using fine couplings and hence the continuous time nature of diffusion processes paths, provides a new way to handle such situations. Note that our approach could lead to further developments. Indeed, while the periodic assumption is natural for applications, it would be interesting to know whether the result remains true without this assumption. In order to generalize our approach to non-periodic coefficients, one only has to prove that there exists two positive constants $\Pi>0$ and $C>0$ such that
\begin{align*}
\left\|Q_{s,t+\Pi}(\mathbf{1}_D)\right\|_\infty \leq C \left\|Q_{s+\Pi,t+\Pi}(\mathbf{1}_D)\right\|_\infty,\ \forall s\leq t.
\end{align*}
However, this apparently simple inequality is related to complex Harnack type inequalities that have not been developed yet. Of course, this inequality is fulfilled in the time periodic framework. We emphasize that the periodicity assumption can be a little bit relaxed, since our result also applies to time-inhomogeneous diffusion processes which are regular time-changes of periodic processes.

For the situation without hard killing, we refer the reader to Del Moral and Miclo~\cite{DelMoral-Miclo2002,DelMoral2003}. For a general account on the time-homogeneous setting, we refer the reader to the notion of quasi-stationary distribution (see the surveys of M\'el\'eard and Villemonais~\cite{Meleard2011} and ofvan Doorn and Pollett~\cite{vanDoornPollett2013}) and, for the particular case of diffusion processes, to the recent results of Knobloch and Partzsch~\cite{KnoblochPartzsch2010}, of Cattiaux and M\'el\'eard\cite{CattiauxMeleard2010}, and of the pioneering works of Pinsky~\cite{Pinsky1985} and of Gong et al.~\cite{Gong1988}. In a recent result, Champagnat and Villemonais~\cite{champagnat-villemonais-15} also provide an abstract necessary and sufficient criterion for the uniform exponential convergence of the conditioned semi-group to a unique quasi-stationary distribution.

The second aim of this paper is to prove a uniform approximation result for the conditioned distribution~\eqref{eq:the-conditioned-distribution}, based on a Fleming-Viot type interacting particle system. More precisely, we consider a particle system of fixed size $N\geq 2$ which evolves as follows. The $N$ particles start in $D$ and evolve as independent copies of $\cal Z$. When one of them is killed, it undergoes a rebirth: it is placed at the position of one other particle chosen uniformly among the remaining ones. Then the particles evolve as independent copies of $\cal Z$, until on of them killed and so on. This type of processes have been introduced independently by Burdzy, Holyst, Ingermann and March in the exploratory paper~\cite{Burdzy1996} and Del Moral and Miclo in~\cite{DelMoral-Miclo2000}. Our main result is that the empirical distribution of the process converges uniformly in time to the conditional distribution~\eqref{eq:the-conditioned-distribution}.  The stability analysis of the discrete time version of these particle absorption processes 
with hard and soft obstacles is developed by Del Moral and Doucet in~\cite{DelMoral-Doucet2004} and by Del Moral and Guionnet in~\cite{DelMoral-Guionnet2000}.

Let us finally mention that adding a soft killing to the process strongly modifies its behavior between two times $s$ and $s+t$ (where $s,t\geq 0$ are fixed) when it is conditioned to not be killed before an arbitrarily large time horizon $T>0$. This is not the case for the non-conditioned process : adding a soft killing only modifies the coupling estimates obtained for the non-conditioned process (between fixed times $s$ and $s+t$) by a constant factor. However, this is not true for the process conditioned to not be killed before time $T$, with arbitrary large values of $T$. Indeed, the probability of not being killed up to time $T$ without soft killing can not be easily related to the same quantity for the process with soft killing, but in the case of a constant (in space) killing rate. Since this quantity plays a fundamental role in the behavior (between times $s$ and $s+t$) of the process conditioned not be killed before time $T$, we have to take into account the killing rate function (even if it is uniformly bounded) along the whole proof. This is allowed by the key ingredient of our approach : the coupling construction and the gradient estimates of~\cite{Priola2006} and the non-degeneracy to the boundary result of~\cite{Villemonais2013}, which are obtained for diffusion processes with soft killing.

\bi In Section~\ref{sec:main-result}, we state our main result, which provides a sufficient criterion ensuring that the mixing property~\eqref{eq:the-mixing-property} holds for time inhomogeneous diffusion processes with both soft and hard
killings. The main tools of the proof, developed in Section~\ref{sec:strong-mixing-property}, are a tightness result for conditional distributions (see Villemonais~\cite{Villemonais2013}) and a coupling between diffusion processes 
(see Priola and Wang~\cite{Priola2006}).

\me In Section~\ref{sec:approximation-method-and-tightness}, we prove an interesting consequence of our main result. Namely, we show that our criterion also implies that the approximation method developed in~\cite{Villemonais_ESAIM_2014} converges uniformly in time to the conditional distribution of the solution to the stochastic differential equation~\eqref{eq:the-eds}.

\section{Main result}
\label{sec:main-result}
Let $\phi_D:D\mapsto \mathbb{R}_+$ denote the Euclidean
 distance to the boundary $\partial D$:
 \begin{equation*}
   \phi_D(x)=d(x,\partial D)=\inf_{y\in\partial D} |x-y|,
 \end{equation*}
 $|\cdot|$ being the Euclidean norm. According to \cite[Chapter 5, Section 4]{Delfour2001}, we can fix $a>0$ small enough so that
$\phi_D$ is of class $C_b^2$ on the boundary's neighbourhood $D^a\subset D$ defined by
 \begin{equation*}
   D^a=\{x\in D\text{ such that }\phi_D(x)<a\}.
 \end{equation*}

\bi \textbf{Assumption (H).}
   We assume that
   \begin{enumerate}
   \item 
           there exists a constant
           $c_0>0$ such that
           \begin{equation*}
             c_0|y| \leq |\sigma(t,x) y|,\ \forall (t,x,y)\in [0,+\infty[\times D\times \mathbb{R}^d.
           \end{equation*}
     \item there exist two measurable functions $f:[0,+\infty[\times
         D^a\rightarrow \mathds{R}_+$ and $g:[0,+\infty[\times
             D^a\rightarrow \mathds{R}$ such that $\forall
             (t,x)\in [0,+\infty[\times D^a$,
    \begin{equation*}
      \sum_{k,l}\frac{\partial \phi_D}{\partial x_k}(x) \frac{\partial
        \phi_D}{\partial x_l}(x)
      [\sigma\sigma^*]_{kl}(t,x)=f(t,x)+g(t,x),
    \end{equation*}
    and such that
    \begin{enumerate}
    \item $f$ is of class $C^1$ in time and of class $C^2$ in
      space, and the successive derivatives of $f$ are uniformly
      bounded,
    \item there exists a positive constant $k_g>0$ such that, for all
      $(t,x)\in[0,+\infty[ \times D^a$,
      \begin{equation*}
	|g(t,x)|\leq k_g\phi_D(x),
      \end{equation*}
    \end{enumerate}
   \end{enumerate}
   
  \me
The first point of Assumption (H) is a classical ellipticity assumption. The second point means that $\sum \d_k \phi_D \d_l\phi_D [\sigma\sigma^*]_{kl}$ can be approximated by a smooth function near the boundary, the error term being bounded by $k_g\phi_D$. In particular, it is satisfied as soon as $\sigma$ and $\phi_D$ are sufficiently smooth in a neighbourhood $D^a$ of $\d D$. More comments on this criterion and examples of processes satisfying assumption (H) are given after Theorem~\ref{thm:strong-mixing}.

  \begin{theorem}
    \label{thm:strong-mixing}
    Assume that assumption (H) is satisfied. Then there exist two constants $C>0$ and $\gamma>0$
    such that, for all $0\leq s\leq t$, 
    \begin{equation*}
      \sup_{\mu_1,\mu_2\in{\cal M}_1(D)}\left\|\frac{\mu_1 Q_{s,t}}{\mu_1
       Q_{s,t} \mathbf{1}_D}-\frac{\mu_2 Q_{s,t}}{\mu_2 Q_{s,t}
        \mathbf{1}_D}\right\|_{TV}\leq C e^{-\gamma (t-s)},
     \end{equation*}
     where ${\cal M}_1(D)$ is the set of probability measures on $D$.
  \end{theorem}

 \medskip\noindent Before turning to the proof of Theorem~\ref{thm:strong-mixing} in Section~\ref{sec:strong-mixing-property}, we give in Remark~\ref{rem:1} four simple  examples of diffusion processes that enter our settings but are not covered by existing results. Then we present in the next section an interesting consequence of Theorem~\ref{thm:strong-mixing} for an approximation method based on a Fleming-Viot type interacting particle system. Note that the approximation method described in the next section is itself used in the proof of Theorem~\ref{thm:strong-mixing}, through the tightness results obtained in~\cite{Villemonais2013}.

\bigskip
\begin{remark}
\label{rem:1}
The second point of assumption~(H) is the only differentiability assumption which is required on the coefficient of the SDE \eqref{eq:the-eds}. We emphasize that it only requires differentiability in an arbitrary small neighbourhood of the boundary $\d D$. Note also that it is fulfilled as soon as the boundary is of class $C^3$ and $\sigma\sigma^*$ is of class $C^1$ in time and $C^2$ in space in the chosen neighbourhood of the boundary. Here are simple and quite natural examples that are not covered by existing results but  enter our settings.
\begin{enumerate}
\item Consider the time-homogeneous setting, where $D$ is the unit ball of $\R^d$, $b(t,x)=0$ and $\sigma(t,x)=(1+\phi_D(x))I_d$ if $x\in D$ and $\sigma(t,x)=I_d$ otherwise. This corresponds, informally, to a standard Brownian motion whose variance increases when the path approaches the center of the domain $D$. In this case, $\sigma$ clearly fulfils the first point of assumption (H), while $\phi_D$ and $\sigma$ are of class $C^\infty$  in $D_{\frac{1}{2}}$ so that the second point of the assumption is also fulfilled and Theorem~\ref{thm:strong-mixing} holds. Note that, since $\sigma$ is not of class $C^1$ in the whole domain $D$, no existing result applies to this simple situation.
\item Consider the situation where $D$ is the unit ball of $\mathbb{R}^2$, $\sigma(t,x)=I_d$ and $b(t,x)=\varphi_D(x)x$ for all $x\in D$. This clearly fulfils Assumption (H), so that Theorem~\ref{thm:strong-mixing} applies. However, no existing result in the literature covers this simple case. Indeed, the general result proven in~\cite{KnoblochPartzsch2010} only applies in $\R^d$ for $d\geq 3$. The studies that succeeded in studying the conditional limit of $2$-dimensional processes (see~\cite{CattiauxMeleard2010} and~\cite{Gong1988}) require at least $C^1$ regularity for both $\sigma$ and $b$. But, in our example, this last function is not $C^1$ in the whole domain $D$.
\item Consider the situation where $D$ is the unit disc of $\R^d$, $\sigma(t,x)=I_d$ and $b(t,x)=R_t$ for $x\in D$, where $R_t$ denotes any time periodic vector. In this situation, all the coefficients are $C^\infty$ in space and $\sigma$ is $C^\infty$ in time. Since the coefficients are time periodic,  we deduce that Theorem~\ref{thm:strong-mixing} applies. However, because of the time-inhomogeneity of $b(t,x)$, no existing continuous time result covers this situation. Also, since the probability of extinction between two successive periods is not bounded away from $1$, no existing discrete time result applies.
\item In fact, our result do not require the diffusion coefficient $\sigma$ to be differentiable at all. For instance, consider the situation where $D$ is the unit ball of $\R^d$ and let $h:\R^d\rightarrow[0,1/2]$ be any positive Lipschitz function. Then the diffusion process solution to the SDE~\eqref{eq:the-eds} with $b(t,x)=0$ and $\sigma(t,x)=(1+h(x)\phi_D(x))I_d$ enters our setting while $\sigma(t,x)$ is \textit{a priori} only Lipschitz in $D$. This is, up to our knowledge, the first result covering such an example.
\end{enumerate}
\end{remark}

 \section{Uniform convergence of a Fleming-Viot type approximation method to the conditional distribution}
 \label{sec:approximation-method-and-tightness}

We present in this section an interesting consequence of the mixing property for
 the approximation method developed in~\cite{Villemonais_ESAIM_2014} : assuming that Assumption~(H) holds, we prove that the
 approximation converges
 uniformly in time.  The particle approximation method has been introduced by Burdzy,
Holyst, Ingerman and March~\cite{Burdzy1996} for standard Brownian
motions, and studied later by Burdzy, Ho{\l}yst and
March~\cite{Burdzy2000} and by Grigorescu and
Kang~\cite{Grigorescu2004} for Brownian motions, by Del Moral and
Miclo~\cite{DelMoral-Miclo2000,DelMoral2003} and by Rousset~\cite{Rousset2006} for
jump-diffusion processes with smooth killings, Ferrari and Mari\`c~\cite{Ferrari2007}
for Markov processes in countable state spaces,
in~\cite{Villemonais2010} for diffusion processes and
in~\cite{Villemonais_ESAIM_2014} for general Markov processes. 
The discrete time version of these interacting particle models on general measurable spaces, including approximations of non absorbed trajectories in terms of genealogical trees is developed in~\cite{DelMoral-Doucet2004,DelMoral-Guionnet2000,DelMoral-Miclo2000b}. For a more detailed discussion, including applications of these discrete generation particle techniques in advanced signal processing, statistical machine learning, and quantum physics, we also refer the reader to the recent monograph~\cite{dp-2013}, and the references therein.

\me
The approximation method is based on a sequence of Fleming-Viot type
interacting particle systems whose associated sequence of empirical
distributions converges to the conditioned
distribution~\eqref{eq:the-conditioned-distribution} when the
number of particles tends to infinity. More precisely, fix $N\geq 2$ and let us define
the Fleming-Viot type interacting particle system with $N$
particles. The system of $N$ particles
$(X^{1,N}_{s,t},...,X^{N,N}_{s,t})_{t\geq s}$ starts from an
initial state $(X^{1,N}_{s,s},...,X^{N,N}_{s,s})\in D^N$, then:
 \begin{itemize}
   \item The particles evolve as $N$ independent copies of ${\cal
     Z}^{X^{i,N}_{s,s}}_{s,\cdot}$ until one of them, say $X^{i_1,N}$, is killed. The
     first killing time is denoted by $\tau^{N}_1$. We emphasize that
     under our hypotheses, the particle killed at time $\tau^{N}_1$ is
     unique~\cite{Villemonais_ESAIM_2014}.
   \item At time $\tau^{N}_1$, the particle $X^{i_1,N}$ jumps on
     the position of an other particle, chosen uniformly among the
     $N-1$ remaining ones. After this operation, the position
     $X^{i,N}_{s,\tau_1^N}$ is in $D$, for all $i\in\{1,...,N\}$.
   \item Then the particles evolve as $N$ independent copies of ${\cal
     Z}^{X^{i,N}_{s,\tau_1^N}}_{\tau_1^N,\cdot}$ until one of them, say $X^{i_2,N}$, is
     killed. This second killing time is denoted by $\tau^{N}_2$. Once
     again, the killed particle is uniquely determined and the $N-1$ other particles are in $D$.
   \item At time $\tau^{N}_2$, the particle $X^{i_2,N}$ jumps on the position of one particle, chosen uniformly among the $N-1$ particles that are in $D$ at time $\tau_2^N$.
   \item Then the particles evolve as independent copies of ${\cal
     Z}^{X^{i,N}_{s,\tau^N_2}}_{\tau^N_2,\cdot}$ and so on.
 \end{itemize}
 We denote by $0<\tau^{N}_1<\tau^{N}_2<...<\tau^{N}_n<...$ the
 sequence of killing/jump times of the process. By~\cite{Villemonais_ESAIM_2014}, Assumption~(H) implies that
 \begin{equation*}
   \lim_{n\rightarrow\infty} \tau^{N}_n=+\infty,\ \text{almost surely.}
 \end{equation*}
  In particular, the above algorithm defines a Markov process
  $(X^{1,N}_{s,t},...,X^{N,N}_{s,t})_{t\geq s}$. For all $N\geq 2$
  and all $0\leq s\leq t$, we denote by $\mu^{N}_{s,t}$ the
  empirical distribution of $(X^{1,N}_{s,t},...,X^{N,N}_{s,t})$,
  which means that
  \begin{equation*}
    \mu^{N}_{s,t}(\cdot)=\frac{1}{N} \sum_{i=1}^N
    \delta_{X^{i,N}_{s,t}}(\cdot) \in {\cal M}_1(D),
  \end{equation*}
  where ${\cal M}_1(D)$ denotes the set of probability measures on the open set
  $D$.  This branching type particle model with fixed population size is closely related to the class of
   Moran and Fleming-Viot processes arising in measure valued super-processes~\cite{DelMoral-Miclo2000,ek-95,Fleming-viot}.  The main difference 
   with these super-processes comes from the fact that the occupation measures of the system converge to a deterministic limit measure valued process, as the size of the population tends to $\infty$.
These particle absorption models  can also be interpreted as an extended version of the Nanbu type mean field  particle model developed by C. Graham and S. M\'el\'eard~\cite{gm-97} in the context of spatially homogeneous Boltzmann equations. The next results provide an uniform estimate w.r.t. the time horizon.

\begin{theorem}
\label{thm:uniform-convergence}
Assume that Hypothesis (H) holds and
that the family of empirical distributions $(\mu^{N}_{s,s})_{s\geq
  0,\,N\geq 2}$ of the initial distributions of
the interacting particle system described above
is tight in ${\cal M}_1(D)$. Then
\begin{equation*}
  \lim_{N\rightarrow\infty} \sup_{s\geq 0}\sup_{t\in[s,+\infty[} \sup_{f\in{\cal B}_1(D)} \mathbb{E}\left|
      \mu^{N}_{s,t}(f) - \frac{\mu^{N}_{s,s} Q_{s,t} (f)}{\mu^{N}_{s,s} Q_{s,t}(\mathbf{1}_D)} \right|=0.
\end{equation*}
\end{theorem}

\medskip
\begin{proof}
Fix $\epsilon>0$. Our aim is to prove
that there exists $N_{\epsilon}\geq 2$ such that, for all $N\geq
N_{\epsilon}$ and all measurable function $f:D\rightarrow \mathbb{R}$
satisfying $\|f\|_{\infty}\leq 1$, we have
\begin{equation}
  \label{eq:caracterisation-epsilon}
  \sup_{s,t\in[0,+\infty[} \mathbb{E}\left| \mu^{N}_{s,s+t}(f) -
      \frac{\mu^{N}_{s,s} Q_{s,s+t} (f)}{\mu^{N}_{s,s} Q_{s,s+t}(\mathbf{1}_D)}
      \right|\leq \epsilon.
\end{equation}

\bigskip \noindent
Let $\gamma$ be the constant of Theorem~\ref{thm:strong-mixing} and fix
$t_0\geq 1$ such that $2e^{-\gamma(t_0-1)}\leq \epsilon/6$. In a first step, we prove that~\eqref{eq:caracterisation-epsilon} holds for $t\leq t_0$. In a second step we prove that it holds for $t\geq t_0$.

\bi \textit{Step 1: Inequality~\eqref{eq:caracterisation-epsilon} holds for $t\leq t_0$.}\\
 Since the sequence of initial distributions is
assumed to be uniformly tight (w.r.t. the time parameter and the size of the system), there exists $\alpha_1=\alpha_1(\epsilon)>0$ such that, $\forall N\geq 2$,
\begin{align*}
\E(\mu^{N}_{s,s}(D^{\alpha_1}))\leq \frac{\epsilon}{8}.
\end{align*}
Now, since the coefficients of the SDE \eqref{eq:the-eds} and the killing rate $\kappa$
are uniformly bounded, the probability for the process ${\cal Z}^x_{s,\cdot}$ starting from $x\in( D^{\alpha_1})^c$ to be killed after time $s+t_0$ is uniformly bounded below by a positive constant. In other words, the constant $\beta_\epsilon$ defined below is positive :
\begin{equation*}
  \beta_{\epsilon}\stackrel{def}{=}\inf\left\{ Q_{s,s+t_0}\mathbf{1}_D(x),\ s\in[0,+\infty[,\ x\in
    \left(D^{\alpha_{1}}\right)^c\right\}>0.
\end{equation*}
Now, fix $s\geq 0$. For all $t\in[0,t_0]$, we have
\begin{equation*}
  \mathbb{E}\left| \mu^{N}_{s,s+t}(f) - \frac{\mu^{N}_{s,s}
    Q_{s,s+t} (f)}{\mu^{N}_{s,s} Q_{s,s+t}(\mathbf{1}_D)}
  \right|=\int_{{\cal M}^1(D)}\mathbb{E}^{N,s}_{\mu}\left| \mu^{N}_{s,s+t}(f)
    - \frac{\mu^{N}_{s,s} Q_{s,s+t} (f)}{\mu^{N}_{s,s}
      Q_{s,s+t}(\mathbf{1}_D)} \right|
    \,d\P(\mu^N_{s,s}=\mu),
\end{equation*}
where $\E^{N,s}_\mu$ (resp. $\P^{N,s}_\mu$) denotes the expectation (resp. the probability) with respect to the law of the Fleming-Viot system with $N$ particles and initial deterministic empirical distribution $\mu$, whose particles evolve as independent copies of ${\cal Z}_{s,\cdot}$ between their rebirths.

In our diffusion process setting, for any $s\leq t$ and  any probability distribution $\mu$, the probability of not being killed before time $t$, given by $\mu Q_{s,t}\1_D$, is positive. As a consequence, we can make use of Theorem~\cite[Theorem~2.2]{Villemonais_ESAIM_2014} to deduce that
\begin{equation*}
\mathbb{E}^{N,s}_{\mu}\left(\left| \mu^{N}_{s,s+t}(f)
    - \frac{\mu^{N}_{s,s} Q_{s,s+t} (f)}{\mu^{N}_{s,s}
      Q_{s,s+t}(\mathbf{1}_D)} \right|
    \right)
\leq  \frac{2(1+\sqrt{2})}{\sqrt{N}}\sqrt{\E^{N,s}_\mu\left(\frac{1}{\mu^{N}_{s,s} Q_{s,s+t}(\mathbf{1}_D)^2 }\right)}.
\end{equation*}
Using the fact that $\mu^N_{s,s}=\mu$ almost surely under $\P^{N,s}_\mu$, we deduce that
\begin{align*}
\mathbb{E}^N_{\mu}\left(\left| \mu^{N}_{s,s+t}(f)
    - \frac{\mu^{N}_{s,s} Q_{s,s+t} (f)}{\mu^{N}_{s,s}
      Q_{s,s+t}(\mathbf{1}_D)} \right|
    \right)
\leq  \frac{2(1+\sqrt{2})}{\sqrt{N}\mu Q_{s,s+t}(\mathbf{1}_D) }.
\end{align*}
Since $\|f\|_{\infty}\leq 1$, we have $\left| \mu^{N}_{s,s+t}(f) - \frac{\mu^{N}_{s,s} Q_{s,s+t}
  (f)}{\mu^{N}_{s,s} Q_{s,s+t}(\mathbf{1}_D)} \right|\leq 2$ almost
surely and we deduce from the previous inequality that
\begin{align*}
  \mathbb{E}\left| \mu^{N}_{s,s+t}(f) - \frac{\mu^{N}_{s,s} Q_{s,s+t} (f)}{\mu^{N}_{s,s}
    Q_{s,s+t}(\mathbf{1}_D)} \right|
  &\leq \frac{\epsilon}{2} + 2\,
        \P\left(\frac{2(1+\sqrt{2})}{\sqrt{N}\mu^{N}_{s,s} Q_{s,s+t}(\mathbf{1}_D) }\geq
        \frac{\epsilon}{2}\right)\\
  &\leq \frac{\epsilon}{2} +2 \,
     \P\left(
        \mu^{N}_{s,s}Q_{s,s+t}(\mathbf{1}_D) \leq
        \frac{4(1+\sqrt{2})}{\epsilon \sqrt{N}}
      \right)
\end{align*}
But $\mu^{N}_{s,s}Q_{s,s+t}(\mathbf{1}_D) \geq
\mu_{s,s}^{N}\left((D^{\alpha_1})^c\right) \beta_{\epsilon}$ almost surely, thus
\begin{align*}
    \mathbb{E}\left| \mu^{N}_{s,s+t}(f) - \frac{\mu^{N}_{s,s} Q_{s,s+t} (f)}{\mu^{N}_{s,s}
      Q_{s,s+t}(\mathbf{1}_D)} \right|
    &\leq \frac{\epsilon}{2} +2\, \P\left(
      \mu_{s,s}^{N}\left((D^{\alpha_1})^c\right) \leq
      \frac{4(1+\sqrt{2})}{\epsilon \sqrt{N} \beta_{\epsilon}} \right)\\
    &\leq \frac{\epsilon}{2}+2\,
      \P\left( \mu_{s,s}^{N}\left(D^{\alpha_{1}}\right) \geq
      1-\frac{4(1+\sqrt{2})}{\epsilon \sqrt{N} \beta_{\epsilon}} \right)\\
    &\leq
      \frac{\epsilon}{2}+ \frac{2}{1-\frac{4(1+\sqrt{2})}{\epsilon \sqrt{N} \beta_{\epsilon}}}
      \E\left( \mu_{s,s}^{N}\left(D^{\alpha_1}\right) \right)\\
    &\leq \frac{\epsilon}{2} +2\frac{1}{1-\frac{4(1+\sqrt{2})}{\epsilon \sqrt{N} \beta_{\epsilon}}}\, \frac{\epsilon}{8},
\end{align*}
where we used Markov's
inequality. Finally, there exists $N_1=N_1(\epsilon)\geq 2$ such that, $\forall N\geq N_1$,
\begin{equation*}
 \sup_{s\geq 0} \sup_{t\in[0,t_0]}  \mathbb{E}\left| \mu^{N}_{s,s+t}(f) - \frac{\mu^{N}_{s,s} Q_{s,s+t} (f)}{\mu^{N}_{s,s}
      Q_{s,s+t}(\mathbf{1}_D)} \right|
    \leq \epsilon.
\end{equation*}

\bi \textit{Step 2: Inequality~\eqref{eq:caracterisation-epsilon} holds for $t\geq t_0$.}\\
Fix now $t\geq t_0$. We have
\begin{multline*}
  \mathbb{E}\left| \mu^{N}_{s,s+t}(f) - \frac{\mu^{N}_{s,s} Q_{s,s+t} (f)}{\mu^{N}_{s,s}
    Q_{s,s+t}(\mathbf{1}_D)} \right| 
  \ \leq 
  \ \mathbb{E}\left|
  \mu^{N}_{s,s+t}(f) - \frac{\mu^{N}_{s,s+1+t-t_0} Q_{s+1+t-t_0,s+t} (f)}{\mu^{N}_{s,s+1+t-t_0}
    Q_{s+1+t-t_0,s+t}(\mathbf{1}_D)}
  \right|\\
  \ +\ 
  \mathbb{E}\left|
  \frac{\mu^{N}_{s,s+1+t-t_0} Q_{s+1+t-t_0,s+t} (f)}{\mu^{N}_{s,s+1+t-t_0}
    Q_{s+1+t-t_0,s+t}(\mathbf{1}_D)} - \frac{\mu^{N}_{s,s} Q_{s,s+t} (f)}{\mu^{N}_{s,s}
    Q_{s,s+t}(\mathbf{1}_D)}
  \right|.
\end{multline*}
By~\cite[Theorem~3.1]{Villemonais2013}, there exists
$\alpha_2=\alpha_2(\epsilon)>0$ and $N_2=N_2(\epsilon)\geq 2$ such that for all $s\in[0,+\infty[$, $t\geq t_0$ and $N\geq N_2$,
\begin{equation*}
   \E\left( \mu^{N}_{s,s+1+t-t_0}(D^{\alpha_2}) \right)\leq \epsilon.
\end{equation*}
Since $\alpha_1$ and $\alpha_2$ can be chosen arbitrarily small, one can assume
without loss of generality that
$\alpha_1=\alpha_2$. Now, using step 1 for the particle system with initial distribution $\mu^{N}_{s,s+1+t-t_0}$ and the Markov property for the particle system, we deduce that, for all $N\geq 2$,
\begin{equation*}
  \mathbb{E}\left| \mu^{N}_{s,s+t}(f) - \frac{\mu^{N}_{s,s+1+t-t_0} Q_{s+1+t-t_0,s+t}
    (f)}{\mu^{N}_{s,s+1+t-t_0} Q_{s+1+t-t_0,s+t}(\mathbf{1}_D)} \right|\leq
 \frac{\epsilon}{2} +2 \frac{1}{1-\frac{4(1+\sqrt{2})}{\epsilon \sqrt{N} \beta_{\epsilon}}}\, \frac{\epsilon}{8}.
\end{equation*}
By Theorem~\ref{thm:strong-mixing} applied to the initial distributions $\mu^{N}_{s,s+1+t-t_0}$ and $\mu^N_{s,s} Q_{s,s+1+t-t_0}$, we also have
\begin{equation*}
\mathbb{E}\left|
  \frac{\mu^{N}_{s,s+1+t-t_0} Q_{s+1+t-t_0,s+t} (f)}{\mu^N_{s,s+1+t-t_0}
    Q_{s+1+t-t_0,s+t}(\mathbf{1}_D)} - \frac{\mu^N_{s,s} Q_{s,s+t} (f)}{\mu^{N}_{s,s}
    Q_{s,s+t}(\mathbf{1}_D)}
  \right|
  \leq 2 e^{-\gamma (t_0-1)} = \epsilon/6.
\end{equation*}
We deduce from the two previous inequality that there exists $N_3=N_3({\epsilon})\geq N_2$ such that, $\forall N\geq N_3$,
\begin{equation*}
  \mathbb{E}\left| \mu^{N}_{s,s+t}(f) - \frac{\mu^{N}_{s,s} Q_{s,s+t} (f)}{\mu^{N}_{s,s}
    Q_{s,s+t}(\mathbf{1}_D)} \right| \leq \epsilon.
\end{equation*}

\bi\textit{Conclusion.}\\
Setting $N_{\epsilon}=N_1\vee N_3$, we have proved~\eqref{eq:caracterisation-epsilon},
which concludes the proof of
Theorem~\ref{thm:uniform-convergence}.
\end{proof}

  \section{Proof of the main result}
\label{sec:strong-mixing-property}

 The proof of Theorem
 \ref{thm:strong-mixing} is based on the study of the process conditioned not to be killed before a time horizon $T$ that we let go to infinity : we prove coupling estimates for these processes, uniformly in $T$. More precisely, let us define, for all $0\leq s\leq
 t\leq T$ the linear operator $R_{s,t}^T$ by
 \begin{equation*}
    R_{s,t}^T f(x)=\frac{ Q_{s,t}(f Q_{t,T}\mathbf{1}_D)(x)}{ Q_{s,T}\mathbf{1}_D(x)},
 \end{equation*}
 for all $x\in D$ and any bounded measurable function $f$. Let us
 remark that the value $R_{s,t}^T f(x)$ is the expectation of $f({\cal
   Z}^x_{s,t})$ conditioned to $T<\tau_{\partial}$. Indeed we have
 \begin{eqnarray*}
   \E\left(\left.f({\cal Z}^x_{s,t})\,\right|\, T<\tau_{\partial}\right)
   &=&\frac{\E\left(f({\cal Z}^x_{s,t})\1_{{\cal Z}^x_{s,T}\in D}\right)}{\E\left(T<\tau_{\partial}\right)}\\
   &=&\frac{\E\left(f({\cal Z}^x_{s,t})\E\left(\left.\1_{{\cal Z}^x_{s,T}\in D}\right| {\cal Z}^x_{s,t}\right)\right)}
      {Q_{s,T}\1_D(x)}\\
   &=&\frac{\E\left(f({\cal Z}^x_{s,t}) \E\left(\1_{{\cal Z}^{{\cal Z}^x_{s,t}}_{t,T}\in D}|{\cal Z}^x_{s,t}\right)\right)}
      {Q_{s,T}\1_D(x)},
 \end{eqnarray*}
 by the Markov property. Finally, since $\E\left(\1_{{\cal Z}^{{\cal
       Z}^x_{s,t}}_{s,T}\in D}|{\cal Z}^x_{s,t}\right)=Q_{t,T}\1_D({\cal
   Z}^x_{s,t})$, we get the announced result.

\me  
 For any $T> 0$, the family
 $(R_{s,t}^T)_{0\leq s\leq t\leq T}$ is an evolution operator. Indeed, we have
 for all $0\leq u\leq s\leq t\leq T$
 \begin{equation*}
   R_{u,s}^T (R_{s,t}^T f)(x) = \frac{ Q_{u,s}(R_{s,t}^T f Q_{s,T}\mathbf{1}_D)(x)}{Q_{u,T}\mathbf{1}_D(x)},
 \end{equation*}
 where, for all $y\in D$,
 \begin{equation*}
   R_{s,t}^T f(y) Q_{s,T}\mathbf{1}_D(y) = Q_{s,t}(f Q_{t,T}\mathbf{1}_D)(y),
 \end{equation*}
 then
 \begin{eqnarray*}
    R_{u,s}^T R_{s,t}^T f(x) &=& \frac{ Q_{u,s}(Q_{s,t}(f
     Q_{t,T}\mathbf{1}_D))(x)}{ Q_{u,T}\mathbf{1}_D(x)}\\
                             &=& \frac{Q_{u,t}(f
     Q_{t,T}\mathbf{1}_D)(x)}{ Q_{u,T}\mathbf{1}_D(x)}= R_{u,t}^T f(x),
 \end{eqnarray*}
 where we have used that $(Q_{s,t})_{s\leq t}$ is an evolution operator.

 In order to prove the exponential mixing property of Theorem
 \ref{thm:strong-mixing}, we need the following lemma, whose proof
 is postponed to the end of this section (see Subsection~\ref{sec:proof-of-key-lemma}). Its proof is based on a coupling construction and on a gradient estimate both obtained in~\cite{Priola2006}. We will also use a non-degeneracy to the boundary result for the conditional distribution proved in~\cite{Villemonais2013}.
 \begin{lemme}
   \label{lem:strong-mixing-for-R}
   There exists a positive constant $\beta'>0$ and a family of probability measures $(\eta^{x_1,x_2}_{s})_{s,x_1,x_2}$  such that, for all $0\leq
   s\leq T-\Pi-2$, we have
   \begin{equation*}
      R_{s,s+1}^T f(x_i) \geq \beta' \eta^{x_1,x_2}_{s}(f),\ i=1,2,
   \end{equation*}
   for all $(x_1,x_2)\in D\times D$ and any non-negative measurable function $f$.
 \end{lemme}

 \bi For any orthogonal probability measures $\mu_1,\mu_2$ on
 $D$, we have 
 \begin{align*}
   \|\mu_1 R_{s,s+1}^T-\mu_2 R_{s,s+1}^T \|_{TV} 
    &=\sup_{f\in {\cal B}_1(D)} |\mu_1 R_{s,s+1}^T(f)-\mu_2 R_{s,s+1}^T(f)|\\
    \leq& \sup_{f\in {\cal B}_1(D)}
   \int_{D\times D} \left|R_{s,s+1}^T f(x) -R_{s,s+1}^T f(y)\right|\ (\mu_1\otimes\mu_2)(dx,dy),
 \end{align*}
 where ${\cal B}_1(D)$ denotes the set of measurable functions $f$
 such that $\|f\|_{\infty}\leq 1$, and $\|\cdot\|_{TV}$ the total
 variation norm for signed measures.  For any $x,y\in D\times D$ and any $f\in{\cal B}_1(D)$, we have, for all $s\leq T-\Pi-2$,
 \begin{align*}
   \left|R_{s,s+1}^T f(x) -R_{s,s+1}^T f(y)\right|
      &=\left|\left(R_{s,s+1}^T f(x)-\beta'\eta_s^{x,y}(f)\right)
                                        -\left( R_{s,s+1}^T f(y)-\beta'\eta_s^{x,y}(f)\right)\right|\\
	  &\leq \left|R_{s,s+1}^T f(x)-\beta'\eta_s^{x,y}(f)\right|
                                        +\left| R_{s,s+1}^T f(y)-\beta'\eta_s^{x,y}(f)\right|\\
 \end{align*}
But, by Lemma~\ref{lem:strong-mixing-for-R}, $\delta_x R_{s,s+1}^T-\beta'\eta_s^{x,y}$ is a non-negative measure with total mass smaller than $1-\beta'$, so that
\begin{align*}
0\leq \left|R_{s,s+1}^T f(x)-\beta'\eta_s^{x,y}(f)\right|\leq 1-\beta'.
\end{align*}
We obtain the same inequality for $\left| R_{s,s+1}^T f(y)-\beta'\eta_s^{x,y}(f)\right|$, so that
\begin{align*}
\left|R_{s,s+1}^T f(x) -R_{s,s+1}^T f(y)\right|&\leq 2(1-\beta').
\end{align*}
 Since $\mu_1$ and $\mu_2$ are assumed to be orthogonal probability
 measures, we have $\|\mu_1-\mu_2\|_{TV}=2$, so that
 \begin{equation*}
   \|\mu_1 R_{s,s+1}^T-\mu_2 R_{s,s+1}^T \|_{TV}\leq (1-\beta') \|\mu_1-\mu_2\|_{TV}.
 \end{equation*}
 If $\mu_1$ and $\mu_2$ are two different but not orthogonal
 probability measures, one can apply the previous result to the
 orthogonal probability measures
 $\frac{(\mu_1-\mu_2)_+}{(\mu_1-\mu_2)_+(D)}$ and
 $\frac{(\mu_1-\mu_2)_-}{(\mu_1-\mu_2)_-(D)}$. Then
 \begin{multline*}
   \left\|\frac{(\mu_1-\mu_2)_+}{(\mu_1-\mu_2)_+(D)} R_{s,s+1}^T-\frac{(\mu_1-\mu_2)_-}{(\mu_1-\mu_2)_-(D)} R_{s,s+1}^T \right\|_{TV}\\
   \leq (1-\beta') \left\|\frac{(\mu_1-\mu_2)_+}{(\mu_1-\mu_2)_+(D)}-\frac{(\mu_1-\mu_2)_-}{(\mu_1-\mu_2)_-(D)}\right\|_{TV}.
 \end{multline*}
But
 $(\mu_1-\mu_2)_+(D)=(\mu_1-\mu_2)_-(D)$ since $\mu_1(D)=\mu_2(D)=1$, then,
 multiplying the obtained inequality by $(\mu_1-\mu_2)_+(D)$, we
 deduce that
 \begin{eqnarray*}
   \|(\mu_1-\mu_2)_+ R_{s,s+1}^T-(\mu_1-\mu_2)_- R_{s,s+1}^T \|_{TV}
   &\leq& (1-\beta') \|(\mu_1-\mu_2)_+-(\mu_1-\mu_2)_-\|_{TV}.
 \end{eqnarray*}
But $(\mu_1-\mu_2)_+-(\mu_1-\mu_2)_-=\mu_1-\mu_2$, so that
 \begin{equation*}
   \|\mu_1 R_{s,s+1}^T-\mu_2 R_{s,s+1}^T \|_{TV}\leq (1-\beta') \|\mu_1-\mu_2\|_{TV}.
 \end{equation*}
 In particular, using the evolution operator property of $(R_{s,t}^T)_{s,t}$, we deduce that
 $$
 \begin{array}{l}
   \|\delta_x R_{0,T-\Pi-2}^T - \delta_y R_{0,T-\Pi-2}^T\|_{TV}\\
   \\
   =\left\|\delta_x R^T_{0,T-\Pi-3} R_{T-\Pi-3,T-\Pi-2}^T - \delta_y R_{0,T-\Pi-3}^T R_{T-\Pi-3,T-\Pi-2}^T\right\|_{TV}\\
        \\   
        \leq (1-\beta')\|\delta_x R_{0,T-\Pi-3}^T - \delta_y R_{0,T-\Pi-3}^T\|_{TV}
           \leq 2 (1-\beta')^{[T-\Pi-2]},
 \end{array}
 $$
 where $[T-\Pi-2]$ denotes the integer part of
 $T-\Pi-2$. Theorem~\ref{thm:strong-mixing} is thus proved for any pair
 of probability measures $(\delta_{x},\delta_y)$, with $(x,y)\in
 D\times D$, for a good choice of $C$ and $\gamma$ which are now assumed to be fixed.

 \bigskip \noindent Let $\mu$ be a probability measure on $D$ and $y\in D$. We have
 \begin{align*}
   \left|\mu Q_{0,T}(f)-\mu Q_{0,T}(\mathbf{1}_D)\frac{\delta_y Q_{0,T}(f)}{\delta_y
     Q_{0,T}(\mathbf{1}_D)}\right|
   &=\left|\int_D Q_{0,T} f(x)
    - \delta_x Q_{0,T}(\mathbf{1}_D)\frac{\delta_y Q_{0,T}(f)}{\delta_y
     Q_{0,T}(\mathbf{1}_D)}d\mu(x)\right|\\
    &\leq \int_D C e^{-\gamma T} \delta_x Q_{0,T}(\mathbf{1}_D) d\mu(x),
 \end{align*}
 by Theorem~\ref{thm:strong-mixing} for Dirac initial measures that we just proved. Dividing by ${\mu
   Q_{0,T}(\mathbf{1}_D)}=\int_D \delta_x Q_{0,T}(\mathbf{1}_D)
 d\mu(x)$, we deduce that
 \begin{equation*}
   \left|\frac{\mu Q_{0,T}(f)}{\mu Q_{0,T}(\mathbf{1}_D)}-\frac{\delta_y Q_{0,T}(f)}{\delta_y
     Q_{0,T}(\mathbf{1}_D)}\right|\leq C e^{-\gamma T},
 \end{equation*}
 for any $f\in{\cal B}_1$. The same procedure, replacing $\delta_y$ by
 any probability measure, leads us to
 Theorem~\ref{thm:strong-mixing}.

  \subsection{Proof of Lemma~\ref{lem:strong-mixing-for-R}}
  \label{sec:proof-of-key-lemma}
 In Subsection~\ref{subsec:coupling}, we present a
  coupling construction for multi-dimensional time-inhomo\-geneous diffusion processes obtained in~\cite{Priola2006}. In
  Subsection~\ref{subsec:intermediate-results}, we derive from
  this coupling and from the non-degeneracy result for conditional distribution obtained in~\cite{Villemonais2013} two intermediate results which are key steps for
  the proof of Lemma~\ref{lem:strong-mixing-for-R}. The first result
  (Lemma~\ref{lem:Q-s-t-1-Q-s-t-1-max}) concerns the existence of a
path $(x_{s,t})_{0\leq s\leq t}$ and a constant $r_0>0$ such that
 \begin{equation*}
   \inf_{x\in B(x_{s,t},r_0)} Q_{s,t}\mathbf{1}_D(x)\geq
   \frac{1}{2}\|Q_{s,t}\mathbf{1}_{D}\|_{\infty},\ \forall 0\leq s+\Pi+1\leq t,
 \end{equation*}
 where we recall that $\Pi$ is the time-period of the coefficients $\sigma$ and $b$.
 The second result (Lemma~\ref{lem:existence-of-mu-s}) states the
 existence of a constant $\beta>0$ and a family of probability
 measures $(\nu^{x_1,x_2}_s)_{s\geq 0,(x_1,x_2)\in D\times D}$ such
 that, for all $s\geq 0$, all $(x_1,x_2)\in D\times D$ and any
 non-negative measurable function $f$,
 \begin{equation*}
  \frac{Q_{s,s+1} f(x_i)}{Q_{s,s+1}\mathbf{1}_D(x_i)} \geq \beta \nu^{x_1,x_2}_s(f), \ i=1,2.
 \end{equation*}
We conclude the proof of Lemma~\ref{lem:strong-mixing-for-R} in
Subsection~\ref{subsec:conclusion-of-the-proof}, showing that
Lemma~\ref{lem:Q-s-t-1-Q-s-t-1-max} and
Lemma~\ref{lem:existence-of-mu-s} imply a uniform coupling rate for the process conditioned not be killed up to an arbitrary time horizon $T>0$.

 \subsubsection{Coupling}
 \label{subsec:coupling}
 In the following proposition, we state the existence of a coupling for multi-dimensional
 time-inhomogeneous diffusion processes. The result also provides a bound for the coupling probability that will be very useful in the next subsections.
 \begin{proposition}
   \label{prop:coupling}
   For all $s\geq 0$ and all $(y_1,y_2)\in D\times D$, there exists a
   diffusion process $(Y^1_{s,t},Y^2_{s,t})_{t\geq s}$ such that
   \begin{enumerate}
     \item $(Y^1_{s,t})_{t\geq s}$ has the same law
       as $({\cal Z}^{y_1}_{s,t})_{t\geq s}$,
     \item $(Y^2_{s,t})_{t\geq s}$ has the same law
       as $({\cal Z}^{y_2}_{s,t})_{t\geq s}$;
     \item $Y^1_{s,t}$ and $Y^2_{s,t}$ are equal almost surely after
       the coupling time 
       \begin{equation*}
         T^s_c=\inf\{t\geq 0,\,Y^1_{s,t}=Y^2_{s,t}\},
       \end{equation*}
       where $\inf \emptyset =+\infty$ by convention.
     \item There exists a constant $c>0$ which doesn't depend on $s,t$
       such that
       \begin{equation*}
         \P^{y_1,y_2}(t <\tau^1_{\partial} \vee \tau^2_{\partial}\text{ and }
         T^s_c> t\wedge \tau^1_{\partial}\wedge \tau^2_{\partial})\leq
         \frac{c|y_1-y_2|}{\sqrt{1\wedge(t-s)}},
       \end{equation*}
       where $\tau^1_{\partial}$ and $\tau^2_{\partial}$ denote the
       killing times of $Y^1$ and $Y^2$ respectively.
   \end{enumerate}
 \end{proposition}

   The proof of Proposition~\ref{prop:coupling} is given in~\cite{Priola2006} for time-homogeneous diffusion processes, though a careful check of the arguments shows that the authors do not use the time-homogeneity of the coefficients to derive the existence and other properties of the coupling. 
   We do not write the proof
   in details, but we recall the idea behind the coupling
   construction. The proof of the $4^{th}$ statement of
   Proposition~\ref{prop:coupling} requires fine
   estimates and calculus which are plainly detailed
   in~\cite{Priola2006}. More precisely, the authors obtain this result on page~261. Note that the result is a little bit hidden but is easily obtained by straightforward computations. Indeed, one checks that, in our situation (and using the notation of the cited article), there exist constants $a,b>0$ such that $F_{\delta,\varepsilon}(\rho(x,y))\leq a\rho(x,y)$ and $h_{\varepsilon}(\rho(x,y))\leq b\rho(x,y)$ in equation~\cite[(4.4)]{Priola2006}.

   Let us nom briefly present the construction of the coupling. By Assumption~(H), there exists
   $\lambda_0>0$ such that $\sigma\sigma^*-\lambda_0 I$ is definite
   positive for all $t,x$. Let
   $\sigma_0:=\sqrt{\sigma\sigma^*-\lambda_0 I}$ be the unique
   symmetric definite positive matrix function such that
   $\sigma_0^2=\sigma\sigma^*-\lambda_0 I$. Without loss of
   generality, one can choose $\lambda_0$ small enough so that
   $\sigma_0$ is uniformly positive definite.  We define
   \begin{equation*}
     u(x,y)=\frac{k(|x-y|)(x-y)}{(k(|x-y|)+1)|x-y|}\text{ and
     }C_t(x,y)=\lambda_0\left(I-2
     u(x,y)u(x,y)^*\right)+\sigma_0(t,x)\sigma_0(t,y)^*,
   \end{equation*}
   where $k(r)=((k_0+1)^2 r^2/2\,\vee\,r)^{\frac{1}{4}}$, where $k_0>0$ is taken from~\eqref{eq:assumption1}.  Before the
   coupling time, the coupling process is
   generated by
   \begin{multline*}
     L_t(x,y)=\frac{1}{2}\sum_{i,j=1}^{d}\left\{
     [\sigma(t,x)\sigma(t,x)^*]_{i,j}\frac{\partial^2}{\partial
       x_i\partial x_j} +
     [\sigma(t,y)\sigma(t,y)^*]_{i,j}\frac{\partial^2}{\partial
       y_i\partial y_j}\right.\\
    \left. + 2 [C_t(x,y)]_{i,j}\frac{\partial^2}{\partial
       x_i\partial y_j} \right\}
     +\sum_{i=1}^d\left\{
     (b(t,x))_i\frac{\partial}{\partial x_i} +
     (b(t,y))_i\frac{\partial}{\partial y_i} \right\}.
   \end{multline*}

   \me The coefficients of $L_t$ are continuous and bounded over
   $\mathds{R}^d$, then, for all $s\geq 0$ and any initial position $x=(y_1,y_2)\in D\times
   D$, there exists a not necessarily unique process
   $(X^{x}_{s,t})_{t\geq 0}$ with values in $\mathds{R}^{2d}$ to the
   martingale problem associated with $(L_t)_{t\geq s}$ (see
   \cite[Theorem 2.2, Chapter IV]{Ikeda1989}). We define $Y'^1$ and $Y'^2$ as the two marginal components of $X^x$, so that 
   $X^{x}_{s,t}=(Y'^1_{s,t},Y'^2_{s,t})$ almost surely. We consider the coupling
   time $T'^s_c$ of $X^{x}_{s,\cdot}$, which is defined by
   \begin{equation*}
     T'^s_c=\inf\{t\geq s,\text{ such that }Y'^1_{s,t}=Y'^2_{s,t}\}.
   \end{equation*}
    We define $Y^1$ and $Y^2$ as follows:
   \begin{equation*}
     Y^i_t=\left\lbrace
     \begin{array}{l}
       Y'^i_t,\ t\leq T'^s_c,\\
       Y'^1_t,\ t>T'^s_c.
     \end{array}
     \right.
   \end{equation*}
   Moreover, each marginal process $Y^{i}$, $i=1,2$, is killed either
   when it hits the boundary $\partial D$ or with a rate $\kappa$.

   \subsubsection{Intermediate results}
   \label{subsec:intermediate-results}
   In this section, we prove Lemmas~\ref{lem:Q-s-t-1-Q-s-t-1-max} and~\ref{lem:existence-of-mu-s}, which are
   essential part of the proof of
   Theorem~\ref{thm:strong-mixing}. We recall that $\Pi$
   denotes the time-period of the coefficients of the SDE~\eqref{eq:the-eds}. Note also that
   Lemma~\ref{lem:Q-s-t-1-Q-s-t-1-max} is the only part of the paper which makes use of
   the periodicity assumption.
   
   \begin{lemme}
     \label{lem:Q-s-t-1-Q-s-t-1-max}
     For any $0\leq s<t$, let us denote by $x_{s,t}\in D$ the point at which
     $Q_{s,t}\mathbf{1}_{D}$ is maximal.  There exists a positive
     constant $r_0>0$ (independent of $s$ and $t$) such that
     \begin{equation*}
       \inf_{x\in B(x_{s,t},r_0)} Q_{s,t}\mathbf{1}_D(x)\geq
       \frac{1}{2}\|Q_{s,t}\mathbf{1}_{D}\|_{\infty},\ \forall 0\leq s\leq s+ \Pi+1\leq t,
     \end{equation*}
     where $B(x_{s,t},r_0)$ denotes the
     ball of radius $r_0$ centred on $x_{s,t}$.
   \end{lemme}
   
 \begin{proof}[Proof of Lemma \ref{lem:Q-s-t-1-Q-s-t-1-max}]
   Fix $s\geq 0$ and let $(Y^1_{s,\cdot},Y^2_{s,\cdot})$ be the
   coupling of Proposition \ref{prop:coupling},
   starting from to points $y_1$ and $y_2$ in $D$. From the properties
   (1) and (2) of the proposition, we deduce that, for any measurable
   bounded function $f$ which vanishes outside $D$, we have
   \begin{eqnarray*}
     \left| Q_{s,s+\Pi}f(y_1) - Q_{s,s+\Pi} f(y_2)\right|\leq \E\left| f(Y^1_{s,s+\Pi})-f(Y^2_{s,s+\Pi}) \right|
   \end{eqnarray*}
   where $f(Y^1_{s,s+\Pi})=f(Y^2_{s,s+\Pi})=0$ if $s+\Pi\geq
   \tau^1_{\partial}\vee\tau^2_{\partial}$ and, by property (3) of
   Proposition~\ref{prop:coupling},
   $Y^1_{s,s+\Pi}=Y^2_{s,s+\Pi}$ if $T^s_c\leq s+\Pi\wedge
   \tau^1_{\partial} \wedge \tau^2_{\partial}$. Thus we have
   \begin{eqnarray}
     \left| Q_{s,s+\Pi}f(y_1) - Q_{s,s+\Pi} f(y_2)\right|
     &\leq&
       \|f\|_{\infty} \P\left(s+\Pi<\tau^1_{\partial}\vee\tau^2_{\partial}\text{ and }
       T^s_c>(s+\Pi)\wedge\tau^1_{\partial}\wedge\tau^2_{\partial}\right) \nonumber\\
       \label{eq:difference-estimateby-coupling}
     &\leq&\|f\|_{\infty} \frac{c|y_1-y_2|}{\sqrt{1\wedge\Pi}},
   \end{eqnarray}
   by the fourth property of Proposition~\ref{prop:coupling}.
 
   \bi
   By
   Proposition~\cite[Theorem~4.1]{Villemonais2013} with $\epsilon=1/2$ and $t_0=1$,
   there exists $\alpha_0>0$ such that, for
   all $0\leq s\leq s+\Pi+1\leq t$,
   \begin{equation*}
     Q_{s+\Pi,t}\mathbf{1}_D(x)\leq 2
     Q_{s+\Pi,t}\1_{\left(D^{\alpha_{0}}\right)^c}(x).
   \end{equation*}
    We emphasize that $\alpha_{0}$ does not depend on $s,t$.
   Now, since the coefficients of the
   SDE~\eqref{eq:the-eds} and the killing rate
   $\kappa$ are assumed to be uniformly bounded on $D$, there exists a
   positive constant $c_{\alpha_0}$, such that, for any $t\geq 0$,
   \begin{equation*}
     \inf_{x\in \left(D^{\alpha_{0}}\right)^c} Q_{t,t+\Pi}\mathbf{1}_D(x)\geq c_{\alpha_0} >0.
   \end{equation*}
   In particular, we have 
   \begin{equation*}
     \1_{\left(D^{\alpha_{0}}\right)^c} \leq \frac{Q_{t,t+\Pi}\mathbf{1}_D}{c_{\alpha_0}}.
   \end{equation*}
   We deduce that, for all $0\leq s\leq s+\Pi+1\leq t$,
   \begin{equation*}
     Q_{s+\Pi,t}\mathbf{1}_D\leq 2
     Q_{s+\Pi,t}\frac{Q_{t,t+\Pi}\mathbf{1}_D}{c_{\alpha_0}},
   \end{equation*}
   so that
  \begin{eqnarray*}
    \| Q_{s+\Pi,t}\left(\mathbf{1}_D\right)\|_{\infty}
       &\leq& \frac{2}{c_{\alpha_0}}\| Q_{s+\Pi,t+\Pi}\mathbf{1}_D\|_{\infty}= \frac{2}{c_{\alpha_0}}\| Q_{s,t}\mathbf{1}_D\|_{\infty},
   \end{eqnarray*}
   by the time-periodicity assumption on the coefficients of the SDE~\eqref{eq:the-eds}. Applying inequality
   \eqref{eq:difference-estimateby-coupling} to
   $f=Q_{s+\Pi,t}\left(\mathbf{1}_D\right)$ and using the evolution operator
   property of $(Q_{s,t})_{s\leq t}$, we deduce that, for all $s\leq
   s+\Pi+1\leq t$,
   \begin{equation*}
     \left| Q_{s,t}\1_D(y_1) - Q_{s,t} \1_D(y_2)\right|\leq
     \frac{2c|y_1-y_2|}{c_{\alpha_0}\sqrt{1\wedge\Pi}}
       \|Q_{s,t}\mathbf{1}_D\|_{\infty}.
   \end{equation*}

   \bigskip \noindent For any $0\leq s\leq s+\Pi+1\leq t$, let $x_{s,t}$ be such that
   $Q_{s,t}\mathbf{1}_D(x_{s,t})=\|
   Q_{s,t}\mathbf{1}_D\|_{\infty}$. We have by the previous inequality,
   \begin{equation*}
     Q_{s,t}\mathbf{1}_D(y)\geq
     \|Q_{s,t}\mathbf{1}_D\|_{\infty}-\frac{2c|x_{s,t}-y|}{c_{\alpha_0}\sqrt{1\wedge
         \Pi}}\| Q_{s,t}\mathbf{1}_D\|_{\infty},\  \forall y\in D.
   \end{equation*}
   Choosing $r_0=\frac{c_{\alpha_0}}{4c}\sqrt{1\wedge\Pi}$, one obtains
   Lemma~\ref{lem:Q-s-t-1-Q-s-t-1-max}.

 \end{proof}

 \begin{lemme}
   \label{lem:existence-of-mu-s}
   There exist a constant $\beta>0$ and a family of probability
   measures denoted by $(\nu^{x_1,x_2}_s)_{s\geq0,(x_1,x_2)\in D\times D}$ such that, for all
   $s\geq 0$, for all $(x_1,x_2)\in D\times D$,    $i\in\{1,2\}$ and for any non-negative measurable
   function $f$,
   \begin{equation*}
 \frac{Q_{s,s+1} f(x_i)}{ Q_{s,s+1}\mathbf{1}_D(x_i)} \geq \beta \nu^{x_1,x_2}_s(f).
   \end{equation*}

   \noindent   Moreover, for any $r_1>0$, we have for all $x\in D$
   \begin{equation*}
     \inf_{s\geq 0,\,(x_1,x_2)\in D\times D}\nu^{x_1,x_2}_s\left(B(x,r_1)\cap D\right)>0.
   \end{equation*}
 \end{lemme}

 \begin{proof}[Proof of Lemma \ref{lem:existence-of-mu-s}]
   Let us first prove that there exist a constant $\rho_0>0$ and a
   fixed point $x_0 \in D$ such that there exists a constant $c_0>0$ such that, for any $(y_1,y_2) \in
   B(x_0,\rho_0)\times B(x_0,\rho_0)$ and any $s\geq 0$, there exists
   a probability measure $\mu^{y_1,y_2}_s$ which fulfills
   \begin{equation}
     \label{eq:inequality-with-mu-first-step}
      \E\left(f({\cal Z}^{y^i}_{s+\frac{2}{3},s+1})\right) \geq c_0\,\mu^{y_1,y_2}_s(f),
   \end{equation}
   for any $i\in\{1,2\}$. Fix $x_0\in D$ and $s\geq 0$. Since the coefficients of the SDE~\eqref{eq:the-eds} and the
   killing rate $\kappa$ are uniformly bounded on $D$, we deduce that, for any value of $\rho>0$ such that $d(\d D,B(x_0,\rho))>0$,
   \begin{equation*}
     \epsilon(\rho) \stackrel{def}{=}\inf_{s\geq 0}\inf_{y_1\in
       B(x_0,\rho)}\E\left(\mathbf{1}_D({\cal
       Z}^{y_1}_{s+\frac{2}{3},s+1})\right)>0.
   \end{equation*}
   Since $\epsilon(\rho)$ is a non-increasing function of $\rho$, one can define a constant $\rho_0>0$ small enough so that $d(\d D,B(x_0,\rho))>0$ and
   \begin{align}
   \label{eq:rho0-choice}
   0<\epsilon(\rho_0)(1-6c\rho_0)-6c\rho_0\leq \epsilon(\rho_0)/2,
   \end{align}
   where $c$ is the positive constant of Proposition~\ref{prop:coupling}.

   Let $y_1,y_2$ be two points of $B(x_0,\rho_0)$ and let
   $(Y^1_{s+\frac{2}{3},\cdot},Y^2_{s+\frac{2}{3},\cdot})$ be the
   coupling of Proposition~\ref{prop:coupling} starting from
   $(y_1,y_2)\in D\times D$ at time $s+\frac{2}{3}$. We define the
   event ${\cal E}$ by
   \begin{equation*}
     {\cal E}=\left\lbrace
     s+1\geq\tau^1_{\partial}\vee\tau^2_{\partial}\text{ or }T^s_c\leq
     (s+1)\wedge\tau^1_{\partial}\wedge\tau^2_{\partial}\right\rbrace,
   \end{equation*}
   where $T^s_c$ is the coupling time of
   Proposition~\ref{prop:coupling}, and $\tau^1_{\partial}$ and
   $\tau^2_{\partial}$ the killing times of $Y^1$ and $Y^2$
   respectively.  By definition of the killing time, $s+1\geq
   \tau^1_{\partial}\vee \tau^2_{\partial}$ implies
   $Y^1_{s+\frac{2}{3},s+1}=Y^1_{s+\frac{2}{3},s+1}=\partial$. Moreover,
   by the coupling property~(3) of
   Proposition~\ref{prop:coupling}, $T^s_c\leq
   (s+1)\wedge\tau^1_{\partial}\wedge\tau^2_{\partial}$ implies
   $Y^1_{s+\frac{2}{3},s+1}=Y^2_{s+\frac{2}{3},s+1}$.  Finally,
   \begin{equation*}
     {\cal E}\subset \left\lbrace Y^1_{s+\frac{2}{3},s+1}=Y^2_{s+\frac{2}{3},s+1}\right\rbrace,
   \end{equation*}
   so that
   \begin{equation*}
     \E\left(f(Y^2_{s+\frac{2}{3},s+1})|{\cal E}\right) =
     \E\left(f(Y^1_{s+\frac{2}{3},s+1})|{\cal E}\right).
   \end{equation*}
   We have then (the first equality being a consequence of
   Proposition~\ref{prop:coupling}~(1)), for any
   measurable function $f$ which vanishes outside $D$,
   \begin{eqnarray*}
     \E\left(f({\cal Z}^{y_i}_{s+\frac{2}{3},s+1})\right)&=&\E\left(f(Y^i_{s+\frac{2}{3},s+1})\right)\\
     &\geq&
     \E\left(f(Y^i_{s+\frac{2}{3},s+1})|{\cal E}\right) \P\left(\cal E\right)\\
     &\geq& \E\left(f(Y^1_{s+\frac{2}{3},s+1})|{\cal E}\right) \P\left(\cal E\right).
   \end{eqnarray*}
    But
   Proposition~\ref{prop:coupling}~(4) implies
   \begin{equation*}
     \P\left({\cal E}\right)\,\geq\ 1-\sqrt{3}\,c|y_1-y_2|\,\geq\ 1-6c\rho_0,
   \end{equation*}
   the last inequality being obtained using the fact that $(y_1,y_2)\in B(x_0,\rho_0)\times B(x_0,\rho_0)$.
We deduce that
   \begin{equation}
     \label{eq:first-appearance-of-mu-s}
     \E\left(f({\cal Z}^{y_i}_{s+\frac{2}{3},s+1})\right)\geq
     \E\left(\mathbf{1}_D(Y^1_{s+\frac{2}{3},s+1})|{\cal E}\right)
     (1-6c\rho_0)\mu^{y_1,y_2}_s(f),\ \forall i=1,2,
   \end{equation}
  where the probability measure $\mu^{y_1,y_2}_s$ on $D$ is defined by
   \begin{equation*}
     \mu^{y_1,y_2}_s(f)=\frac{\E\left(f(Y^1_{s+\frac{2}{3},s+1})|{\cal
         E}\right)}{\E\left(\mathbf{1}_D(Y^1_{s+\frac{2}{3},s+1})|{\cal E}\right)}.
   \end{equation*}
   It only remains to find a lower bound for
   $\E\left(\mathbf{1}_D(Y^1_{s+\frac{2}{3},s+1})|{\cal E}\right)$ to
   conclude that
   \eqref{eq:inequality-with-mu-first-step} holds.  We have
   \begin{align*}
     \E\left(\mathbf{1}_D(Y^1_{s+\frac{2}{3},s+1})|{\cal E}\right)
     &= \frac{1}{\P({\cal
         E})}\E\left(\mathbf{1}_D(Y^1_{s+\frac{2}{3},s+1})\right)
     -\frac{1-\P({\cal E})}{\P({\cal E})}\E\left(\mathbf{1}_D(Y^1_{s+\frac{2}{3},s+1})|{\cal E}^c\right)\\ 
     &\geq \frac{1}{\P({\cal
         E})}\E\left(\mathbf{1}_D({\cal
       Z}^{y_1}_{s+\frac{2}{3},s+1})\right)-\frac{1-\P({\cal E})}{\P({\cal E})}\\ 
     &\geq 
     \E\left(\mathbf{1}_D({\cal
       Z}^{y_1}_{s+\frac{2}{3},s+1})\right)-\frac{6c\rho_0}{1-6c\rho_0}.
   \end{align*}
   Finally, we deduce from \eqref{eq:first-appearance-of-mu-s} and from the definition of $\epsilon(\rho_0)$ that
   that 
   \begin{align*}
     \E\left(f({\cal Z}^{y_i}_{s+\frac{2}{3},s+1})\right)\geq
     \left(\epsilon(\rho_0)-\frac{6c\rho_0}{1-6c\rho_0}\right) (1-6c\rho_0)
     \mu^{y_1,y_2}_s(f),
   \end{align*}
   for any non-negative measurable function which vanishes outside
   $D$. Using inequality~\eqref{eq:rho0-choice}, 
   we deduce that, for all $i\in\{1,2\}$,
   \begin{equation}
      \E\left(f({\cal Z}^{y_i}_{s+\frac{2}{3},s+1})\right)  \geq \frac{\epsilon(\rho_0)}{2}\mu^{y_1,y_2}_s(f).
     \label{eq:inequality-mu-s-end-first-step}
   \end{equation}
   
   \bi 

   \bi Let us now conclude the proof of
   Lemma~\ref{lem:existence-of-mu-s}.  By
  \cite[Theorem~4.1]{Villemonais2013}
   with $\mu=\delta_x$, there exists a constant $\alpha_1>0$ such
   that, for all $s\geq 0$ and all $x\in D$,
   \begin{eqnarray}
     Q_{s,s+\frac{1}{3}}\1_{(D^{\alpha_1})^c}(x)&\geq&
     \frac{1}{2} Q_{s,s+\frac{1}{3}}\mathbf{1}_D(x)\nonumber\\
     &\geq& \frac{1}{2} Q_{s,s+1}\mathbf{1}_D(x).\label{eq:between-s-and-s-plus-1-on-3}
   \end{eqnarray}
   Since the coefficients of the SDE~\eqref{eq:the-eds} and the
   killing rate $\kappa$ are uniformly bounded, we have
   \begin{equation*}
     \epsilon_1\stackrel{def}{=}\inf_{s\geq 0,
       x\in\left(D^{\alpha_1}\right)^c}
     Q_{s+\frac{1}{3},s+\frac{2}{3}}\1_{ B(x_0,\rho_0)}(x)>0.
   \end{equation*}
   In particular, we deduce from
   \eqref{eq:between-s-and-s-plus-1-on-3} that, for all $x\in
   D$,
   \begin{eqnarray}
    Q_{s,s+\frac{2}{3}}\1_{
       B(x_0,\rho_0)}(x)&=& 
     Q_{s,s+\frac{1}{3}}\left(Q_{s+\frac{1}{3},s+\frac{2}{3}}\1_{
       B(x_0,\rho_0)}\right)(x)\nonumber \\
       &\geq& \epsilon_1 
       Q_{s,s+\frac{1}{3}}\1_{\left(D^{\alpha_{1}}\right)^c}(x)\nonumber \\
       &\geq&\frac{\epsilon_1}{2}  Q_{s,s+1}\mathbf{1}_D(x) \label{eq:useful-born-1}
   \end{eqnarray}
   Finally, we have, for all
   $x_1,x_2\in D\times D$, 
   \begin{eqnarray*}
     Q_{s,s+1} f(x_i)&\geq & \int_{B(x_0,\rho_0)}
     Q_{s+\frac{2}{3},s+1} f(y_1) \left[\delta_{x_1}
     Q_{s,s+\frac{2}{3}}\right](dy_1)\\
     &\geq & \frac{1}{
       Q_{s,s+\frac{2}{3}}\1_{B(x_0,\rho_0)}(x_2)}\ \times\\
     & &\ \int_{B(x_0,\rho_0)}
     \int_{B(x_0,\rho_0)} Q_{s+\frac{2}{3},s+1} f(y_1) \left[\delta_{x_1}
     Q_{s,s+\frac{2}{3}}\otimes \delta_{x_2} Q_{s,s+\frac{2}{3}}\right](dy_1,dy_2)\\
     &\geq & \frac{\epsilon_0}{2
       Q_{s,s+\frac{2}{3}}\1_{B(x_0,\rho_0)}(x_2)}\ \times\\
     & &\ \int_{B(x_0,\rho_0)}
     \int_{B(x_0,\rho_0)} \mu^{y_1,y_2}_s(f) \left[\delta_{x_1}
     Q_{s,s+\frac{2}{3}}\otimes \delta_{x_2} Q_{s,s+\frac{2}{3}}\right](dy_1,dy_2)
     \text{ by \eqref{eq:inequality-mu-s-end-first-step}}\\
     &\geq& \epsilon_0\frac{ Q_{s,s+\frac{2}{3}}\1_{B(x_0,\rho_0)}(x_1)}{2}\nu^{x_1,x_2}_s(f),
   \end{eqnarray*}
   where $\nu^{x_1,x_2}_s$ is the probability measure on $D$ defined by
   \begin{equation*}
     \nu^{x_1,x_2}_s(f)=\frac{\int_{B(x_0,\rho_0)}
     \int_{B(x_0,\rho_0)} \mu^{y_1,y_2}_s(f) \left[\delta_{x_1}
     Q_{s,s+\frac{2}{3}}\otimes \delta_{x_2} Q_{s,s+\frac{2}{3}}\right](dy_1,dy_2)}
        { Q_{s,s+\frac{2}{3}}\1_{B(x_0,\rho_0)}(x_1) Q_{s,s+\frac{2}{3}}\1_{B(x_0,\rho_0)}(x_2)}.
   \end{equation*}
   This and Inequality~\eqref{eq:useful-born-1} allow us to
   conclude the proof of the first part of Lemma~\ref{lem:existence-of-mu-s}.

   \bigskip\noindent Fix $r_1>0$ and let us prove the second part of
   the lemma. We have, for all $(y_1,y_2)\in (B(x_0,\rho_0))^2$ and all $x\in D$,
   \begin{align*}
    \mu^{y_1,y_2}_{s}(B(x,r_1))&= \frac{\E\left(\mathbf{1}_{B(x,r_1)}(Y^1_{s+\frac{2}{3},s+1})|{\cal
         E}\right)}
          {\E\left(\mathbf{1}_D(Y^1_{s+\frac{2}{3},s+1})|{\cal
              E}\right)}\\
              &\geq \E\left(\1_{B(x,r_1)}(Y^1_{s+\frac{2}{3},s+1})\mid {\cal E}\right)\,\P({\cal E})\\
          &\geq \E\left(\mathbf{1}_{B(x,r_1)}(Y^1_{s+\frac{2}{3},s+1})\right)-\left(1-\P({\cal E})\right)\\
          &\geq \delta_{y_1}Q_{s+\frac{2}{3},s+1}(B(x,r_1))-6c\rho_0.
   \end{align*}
    We
   emphasize that, because of the boundedness and the regularity of $D$,
   $B(x,r_1)\cap D$ contains a ball of minimal volume uniformly over $x\in D$.  Then, since the coefficients of the
   SDE~\eqref{eq:the-eds} and the killing rate $\kappa$
   are assumed to be uniformly bounded, we have
   \begin{equation*}
     \epsilon_2\stackrel{def}{=}\inf_{s\geq 0,\ x\in D,\ y_1\in B(x,\rho_0)} Q_{s+\frac{2}{3},s+1}\1_{B(x,r_1)}(y_1)>0,
   \end{equation*}
   where we recall that $\rho_0$ is chosen small enough so that $d(\d D,B(x_0,\rho_0))>0$.
    We deduce that
   \begin{equation*}
     \mu_s^{y_1,y_2}(B(x,r_1))\geq \epsilon_2/2.
   \end{equation*}
Finally, by definition of $\nu_s^{x,y}$, we deduce that
   \begin{equation*}
     \nu^{x,y}_s(B(x,r_1))\geq \epsilon_2/2,\ \forall x\in D.
   \end{equation*}
   This concludes the proof of Lemma~\ref{lem:existence-of-mu-s}.
  \end{proof}

 \subsubsection{Conclusion of the proof of Lemma~\ref{lem:strong-mixing-for-R}}
 \label{subsec:conclusion-of-the-proof}

   By Lemma~\ref{lem:existence-of-mu-s}, there exist $\beta>0$ and a
   family of probability measures denoted by $(\nu_s^{x_1,x_2})_{s\geq
     0,\,(x_1,x_2)\in D\times D}$ such that, for any $(x_1,x_2)\in D\times
   D$ and any $s\geq 0$, we have for all $i\in\{1,2\}$
   \begin{equation*}
      Q_{s,s+1}f(x_i)\geq Q_{s,s+1}\mathbf{1}_D(x_i)
     \beta \nu^{x_1,x_2}_{s}(f),
   \end{equation*}
   for any non-negative measurable function $f$.
   Then we have
   \begin{eqnarray*}
     R_{s,s+1}^T f(x_i)&=&\frac{Q_{s,s+1}(fQ_{s+1,T}\mathbf{1}_D)(x_i)}{Q_{s,T}\mathbf{1}_D(x_i)}\\
                  &\geq&\frac{\beta \nu^{x_1,x_2}_{s}(f Q_{s+1,T}\mathbf{1}_D) Q_{s,s+1}\mathbf{1}_D(x_i)}
                             {Q_{s,T}\mathbf{1}_D(x_i)}.
   \end{eqnarray*}
   Since $s+1+\Pi+1\leq T$ by assumption, we deduce from
   Lemma~\ref{lem:Q-s-t-1-Q-s-t-1-max} that there exist $x_{s+1,T}\in D$
   and $r_0>0$ such that
   \begin{equation}
     \label{eq:bound-below-Q-s-T-max}
     \inf_{x\in B(x_{s+1,T},r_0)}
     Q_{s+1,T}\mathbf{1}_D(x)\geq \frac{1}{2}
     \|Q_{s+1,T}\mathbf{1}_D\|_{\infty}.
   \end{equation}
 Now, we define the probability
   measure $\eta^{x_1,x_2}_s$  by
   \begin{equation*}
     \eta^{x_1,x_2}_s(A)\stackrel{def}{=}\frac{\nu^{x_1,x_2}_{s}(A\cap B(x_{s,T},r_0))}{\nu^{x_1,x_2}_{s}(B(x_{s,T},r_0))},\ \forall A\subset D.
   \end{equation*}
   By the second part of
   Lemma~\ref{lem:existence-of-mu-s},
   $\nu^{x_1,x_2}_{s}(B(x_{s,T},r_0))$ is uniformly bounded below by a
   constant $\epsilon>0$ which only depend on $r_0$.
      We deduce that
   \begin{eqnarray*}
     R_{s,s+1}^T f(x_1)&\geq&\frac{\frac{\epsilon}{2} \eta^{x_1,x_2}_{s}(f)
                                 Q_{s,s+1}\mathbf{1}_D(x_1) \|Q_{s+1,T}\mathbf{1}_D\|_{\infty}}
                                {Q_{s,T}\mathbf{1}_D(x_1)}\\
                     &\geq&\frac{\epsilon}{2}\eta^{x_1,x_2}_{s}(f),
   \end{eqnarray*}
   by the Markov property.
   This concludes the proof of Lemma~\ref{lem:strong-mixing-for-R}.

%

\end{document}